\documentclass[reqno]{amsart}
\usepackage{amssymb,amsthm}

\newcommand{\rsp}{\raisebox{0em}[2.3ex][1.3ex]{\rule{0em}{2ex} }}

\newcommand{\fp}{{\mathfrak p}}

\newcommand{\Z}{{\mathbb Z}}

\newcommand{\Q}{{\mathbb Q}}

\newcommand{\fa}{\mathfrak a}

\newcommand{\cR}{{\mathcal R}}

\newcommand{\Cl}{{\operatorname{Cl}}}

\newcommand{\lra}{\longrightarrow}

\newcommand{\Sim}{\v{S}imerka}

\newcommand{\cs}{\operatorname{\check{s}}}

\newtheorem{thm}{Theorem}[section]
\newtheorem{prop}[thm]{Proposition}
\newtheorem{lem}[thm]{Lemma}

\numberwithin{equation}{section}

\title{V\'aclav \v{S}imerka: \\ Quadratic Forms and Factorization}
\author{F. Lemmermeyer}
\email{hb3@ix.urz.uni-heidelberg.de}
\address{M\"orikeweg 1, 73489 Jagstzell, Germany}

\begin{document}

\begin{abstract}
In this article we show that the Czech mathematician V\'aclav
\Sim\ discovered the factoriation of $\frac19 (10^{17}-1)$
using a method based on the class group of binary quadratic 
forms more than 120 years before Shanks and Schnorr developed
similar algorithms. \Sim\ also gave the first examples of what
later became known as Carmichael numbers.
\end{abstract}

\maketitle

\medskip \noindent 
According to Dickson \cite[I. p. 172]{Dick}, the number 
$$ N = 11111111111111111 = \frac{10^{17}-1}9 $$ 
was first factored by Le Lasseur in 1886, and the result was 
published by Lucas in the same year. Actually the factorization 
of $N$ already appeared as a side result in a forgotten memoir 
\cite{Sim1} of V\'aclav\footnote{In his German publications, 
\Sim\ used the germanized name Wenzel instead of V\'aclav.} 
\Sim, in which he presented his ideas on composition of positive 
definite forms, computation of class numbers, and the prime 
factorization of large integers such as $N$.

In fact, consider the binary quadratic form 
$$ Q = (2, 1, 1388888888888889)  $$
with discriminant $\Delta = -N$. If we knew that 
$h = 107019310$ was (a multiple of) the order of $[Q]$ in $\Cl(-N)$, 
then a simple calculation would reveal that
$$ Q^{h/2} \sim (2071723, 2071723, 1341323520), $$
from which we could read off the factorization
$$ N = 2071723 \cdot 5363222357. $$
This idea for factoring integers was later rediscovered
by Daniel Shanks in the 1970s; subsequent work on this idea 
led Shanks to introduce the notion of infrastructure, which 
has played a major role in algorithmic number theory since then.

In \cite{Sim1}, \Sim\ explains Gauss's theory of composition
using the language from Legendre's Th\'eorie des Nombres. The 
rest of his article \cite{Sim1} is dedicated to the calculation 
of the order of a quadratic form in the class group, and an 
application to factoring integers. 

In this article we will review \Sim's work and explain some of his
calculations so that the readers may convince themselves that 
\cite{Sim1} contains profound ideas and important results. 

\section{A Short Biography}

V\'aclav \Sim\ was born on Dec. 20, 1819, in Hochwesseln
(Vysok\'em Vesel\'i). He studied philosphy and theology in 
K\"oniggr\"atz, was ordained in 1845 and worked as a chaplain 
in \v{Z}lunice near Ji\v{c}\'in. He started studying mathematics
and physics in 1852 and became a teacher at the gymnasium of Budweis.
He did not get a permanent appointment there, and in 1862 became
priest in Jen\v{s}ovice near Vusok\'e M\'yto. Today, \Sim\ is
remembered for his textbook on algebra (1863); its appendix contained 
an introduction to calculus and is the first Czech textbook on calculus.
\Sim\ died in Praska\v{c}ka near K\"oniggr\"atz (Praska\v{c}ce u Hradce 
Kr\'alov\'e) on Dec. 26, 1887.

\Sim's contributions to the theory of factoring have not been
noticed at all, and his name does not occur in any history of 
number theory except Dickson's: see \cite[II, p. 196]{Dick} for
a reference to \Sim's article \cite{Sim3}, which deals with the 
diophantine problem of rational triangles. In \cite[III, p. 67]{Dick}, 
Dickson even refers to \cite{Sim1} in connection with the composition
of binary quadratic forms. 

In \cite{Sim2}, \Sim\ gave a detailed presentation of a large
part of Legendre's work on sums of three squares. In \cite{SimFN}, 
\Sim\ proved that 
$7 \cdot 2^{14} + 1 \mid F_{12}$ and $5 \cdot 2^{25} + 1 \mid F_{23}$
(these factors had just been obtained by Pervouchin), 
where $F_n$ denotes the $n$-th Fermat number. 
In \cite{SimF}, \Sim\ listed the Carmichael numbers \cite{CN}
$$ n = 561, 1105, 1729, 2465, 2821, 6601, 8911 $$
long before Korselt \cite{Kors} gave criteria hinting at their
existence and Carmichael \cite{Carm} gave what was believed to be
the first example. All of \Sim's examples are products of three
prime factors, and there are no others below $10\,000$.

For more on \Sim, see \cite{Cupr,Kop,Panek}.

\section{The \v{S}imerka Map}

Let us now present \Sim's ideas from \cite{Sim1} in a modern form. At
the end of this section, we will explain \Sim's language. Let $Q$ be 
a positive definite binary quadratic form with discriminant $\Delta$. 
If $Q$ primitively represents a (necessarily positive) integer $a$, 
then $Q$ is equivalent to a unique form $(a,B,C)$ with $-a < B \le a$. 
Let 
$$ a = p_1^{a_1} \cdots p_r^{a_r} $$
denote the prime factorization of $a$. For each prime $p_j \mid a$,
fix an integer $-p_j < b_j \le p_j$ with $B \equiv b_j \bmod p_j$ and set
$$ s_j = \begin{cases}
             +1 & \text{ if } b_j \ge 0, \\
             -1 & \text{ if } b_j < 0. \end{cases} $$
Thus if $a = Q(x,y)$, then we can define
$$ \cs(Q,a) =  \prod p_j^{s_ja_j}. $$

\medskip\noindent
{\bf Example.} 
The principal form $Q_0 = (1,0,5)$ with discriminant $-20$ represents the
following values:
$$ \begin{array}{c|ccccccc}
  \rsp  a      &    1    &   5     &    6      &    9    
               &    14   &   21  &  21 \\ \hline
  \rsp  Q      & (1,0,5) & (5,0,1) & (6,2,1)   & (9,4,1) 
               & (14,6,1) & (21,8,1) & (21,20,5) \\
  \rsp \cs(a,Q_0) &    1    &   5     & 2 \cdot 3 &   3^2   
               & 2 \cdot 7  & 3 \cdot 7 & 3^{-1} \cdot 7
   \end{array} $$

Forms equivalent to $Q = (2,2,3)$ give us the following values:
  $$  \begin{array}{c|ccccc}
  \rsp  a      &    2    &   3     &    7   &  87 &  87 \\ \hline
  \rsp  Q      & (2,2,3) & (3,-2,2) &  (7,6,2) &  (87,26,2)  & (87,32,3) \\
  \rsp \cs(a,Q_0) &    2    &   3^{-1}   & 7    &  3 \cdot 29 &  3 \cdot 29^{-1}
   \end{array} $$

The ideal theoretic interpretation of the \Sim\ map is the following:
there is a correspondence between binary quadratic forms $Q$ with 
discriminant $\Delta < 0$ and ideals $\fa(Q)$ in a suitable order of the 
quadratic number field $\Q(\sqrt{\Delta}\,)$. Equivalent forms correspond
to equivalent ideals, and integers $a$ represented by $Q$, say $Q(x,y) = a$,
correspond to norms of elements $\alpha \fa(Q)$ via $a = N\alpha/N\fa(Q)$.
Integers represented primitively by $Q$ are characterized by the fact
that $\alpha \in \fa(Q)$ is not divisible by a rational prime number. 
If we fix prime ideals $\fp_j = \fa(Q_j)$ by $\fa(Q_j)$ for 
$Q_j = (p_j, B_j, C)$ with $0 \le B_j \le p_j$ and formally set 
$\fp_j^{-1} = \fa(Q_j')$ with $Q_j' = (p_j, -B_j, C)$, then 
$\cs(a,Q) = p_1^{a_1} \cdots p_r^{a_r}$ is equivalent to 
$(\alpha) = \fp_1^{a_1} \cdots \fp_r^{a_r} \fa(Q)$.

Assume that $a = p_1 \cdots p_r$, and that 
$Q = (a,B,C)$. Then 
$$ (a,B,C) = (p_1,B,p_2\cdots p_rC) \cdot
             (p_2,B,p_1p_3 \cdots p_rC) \cdots
             (p_r,B,p_1 \cdots p_{r-1}C). $$
If we write $b_j \equiv B \bmod 2p_j$ with $-p_j < b_j \le p_j$,
then 
$$ \cs(a,Q) = \cs(p_1,Q_1) \cdots \cs(p_r,Q_r) $$ 
by definition of $\cs$.

We start by showing that the value set of $\cs$ is closed with
respect to inversion. To this end we use the notation 
$(A,B,C)^{-1} = (A,-B,C)$. Then it follows right from the definition
of $\cs$ that if $\cs(a,Q) = r$, then $\cs(a,Q^{-1}) = r^{-1}$.

Now we claim

\begin{lem}\label{Sm}
Let $\Delta$ be a fundamental discriminant. 
Assume that $Q_1(x_1,y_1) = a_1$ and  $Q_2(x_2,y_2) = a_2$, and that
$Q_3 \sim Q_1Q_2$. Then there exist integers $a_3, x_3, y_3$ such that 
$Q_3(x_3,y_3) = a_3$ and $\cs(a_3,Q_3) = \cs(a_1,Q_1) \cdot \cs(a_2,Q_2)$.
\end{lem}

\begin{proof}
Writing
$Q_1 = (a_1,B_1,C_1) = (p_1,B_1,a_1C_1/p_1) \cdots (p_r,B_1,a_1C_1/p_r)$ and 
$Q_2 = (a_2,B_2,C_2) = (q_1,B_2,a_2C_2/q_1) \cdots (q_s,B_2,a_2C_2/q_s)$, 
where $a_1 = p_1 \cdots p_r$ and $a_2 = q_1 \cdots q_s$ are the 
prime factorizations of $a_1$ and $a_2$, we see that it is sufficient 
to prove the result for prime values of $a_1$ and $a_2$. There are 
several cases:
\begin{enumerate}
\item $Q_1 = (p,b_1,c_1)$, $Q_2 = (q,b_2,c_2)$ with $p \ne q$:
      for composing these forms using Dirichlet's method, we choose
      an integer $b$ satisfying the congruences
      $$ b \equiv b_1 \bmod 2p, \quad \text{and} \quad 
         b \equiv b_2 \bmod 2q.  $$  
      Then $Q_1 \sim (p,b,qc')$ and $Q_2 \sim (q,b,pc')$, and we find
      $Q_1Q_2 = (pq,b,c')$ as well as $\cs(pq,Q_1Q_2) = \cs(p,Q_1)\cs(q,Q_2)$
      by the definition of $\cs$.
\item $Q_1 = (p,b_1,c_1)$, $Q_2 = (p,-b_1,c_1) = Q^{-1}$: here Dirichlet
      composition shows $Q_1Q_2 = (1,b_1,pc_1) \sim Q_0$, and since
      $\cs(Q_2) = \cs(Q_1)^{-1}$ we also have  
      $1 = \cs(1,Q_1Q_2) = \cs(p,Q_1)\cs(p,Q_2)$.
\item $Q_1 = (p,b_1,c_1)= Q_2$: if $p \nmid \Delta$, then $p \nmid b_1$,
      and we can easily find an integer $b \equiv b_1 \bmod 2p$ with
      $b^2 \equiv \Delta \bmod 2p_1^2$. But then $Q_1 \sim (p,b,pc')$
      and, by Dirichlet composition, $Q_1^2 = (p^2,b,c')$. As before,
      the definition of $\cs$ immediately shows that 
      $\cs(p^2,Q_1^2) = \cs(p,Q_1)^2$.

      If $p \mid \Delta$ and $p$ is odd, on the other hand, then 
      $p \mid b_1$. Since $\Delta$ is fundamental, the form $Q_1$ is
      ambiguous, hence $Q_1^2 \sim Q_0$. Since $\cs(Q_1) = 1$, the
      multiplicativity is clear.
\end{enumerate}
This completes the proof. 
\end{proof}
  
\begin{prop}
Let $Q_0$ denote the principal form with discriminant $\Delta < 0$. Then
the elements $\cs(a,Q_0)$ form a subgroup $\cR$ of $\Q^\times$.
\end{prop}

\begin{proof}
It remains to show that if $Q$ represents $a$ and $b$, then it
represents $ab$ in such a way that $\cs(ab,Q_0) = \cs(a,Q_0) \cs(b,Q_0)$.
Again we can reduce this to the case of prime values of $a$ and $b$,
and in this case the claim follows from the proof of Lemma \ref{Sm}.
\end{proof}

\begin{prop}
Assume that $a$ is represented properly by $Q$, and that $a'$
is represented properly by $Q'$. If $Q \sim Q'$, then 
$$ \cs(a,Q) \equiv \cs(a',Q') \bmod \cR. $$
\end{prop}

\begin{proof}
Since equivalent forms represent the same integers it is sufficient to 
show that if a form $Q$ properly represents numbers $a$ and $b$, then 
$\cs(a,Q) \equiv \cs(b,Q) \bmod \cR$.

Assume that $Q = (A,B,C)$, and set $\cs(a,Q) = r$ and $\cs(b,Q) = s$.
If $a$ and $b$ are coprime, then $\cs(ab,Q_0) = r \cdot s^{-1} \in \cR$, 
where $Q_0$ is the composition of $Q$ and $Q^{-1}$. This implies the claim.

If $a$ and $b$ have a factor in common, then there is an integer $c$ 
such that $n=ab/c^2$ is represented by $Q_0$ in such a way that 
$\cs(n,Q_0) = r \cdot s^{-1} \in \cR$, and the claim follows as above.
\end{proof}

These propositions show that $\cs$ induces a homomorphism
$$ \cs: \Cl(\Delta) \lra \Q^\times/\cR $$
from the class group $\Cl(\Delta)$ to $\Q^\times/\cR$, which we
will also denote by $\cs$, and which will be called the \Sim\ map.

\begin{thm}
Let $\Delta < 0$ be a fundamental discriminant. Then the \Sim\ map is 
an injective homomorphism of abelian groups.
\end{thm}

\begin{proof}
We have to show that $\cs$ is injective. To this end, let 
$[Q]$ denote a class with $a = \cs(Q) \in \cR$. Then there
is a form $Q_0' = (A,B,C) \sim Q_0$ with $\cs(A,Q_0) = a$.
But then $Q_1 = Q \cdot (A,-B,C)$ is a form equivalent to $Q$ 
with $\cs(Q_1) = 1$. This in turn implies that $Q_1$ represents
$1$, hence is equivalent to the principal form by the classical
theory of binary quadratic forms.
\end{proof}

\Sim's idea is to use a set of small prime numbers $S = \{p_1, \ldots, p_r\}$
which are smaller than $\sqrt{-\Delta/3}$ (and a subset of these if 
$|\Delta|$ is large), find integers $a_j$ primitively represented by $Q$
whose prime factors are all in $S$, and using linear combinations to 
find a relation in $\cR$, which gives him an integer $h$ such that
$Q^h \sim 1$. It is then easy to determine the exact order of $Q$.

\subsection*{\Sim's Language}

\Sim\ denotes binary quadratic forms $Ax^2 + Bxy + Cy^2$ by $(A,B,C)$
and considers forms with even as well as with odd middle coefficients.
The principal form with discriminant $\Delta$ is called an 
end form\footnote{Computing the powers of a form $Q$, one finds
$Q$, $Q^2$, \ldots, $Q^h \sim Q_0$ before everything repeats. The last
form in such a ``period'' of reduced forms is thus always the principal 
form.} (Endform, Schlussform), and ambiguous\footnote{The word
ambiguous was coined by Poullet-Deslisle in the French translation of
Gauss's Disquisitiones Arithmeticae; it became popular after Kummer
had used it in his work on higher reciprocity laws. \Sim\ knew 
Legendre's ``diviseurs quadratiques bifides'' as well as Gauss's 
``forma anceps''.} forms are called middle forms (Mittelformen). 

The subgroup generated by a form $Q$ is called its period, the exponent 
of a form $Q$ in the class group is called the length of its period.
\Sim\ represents a form $f = (A,B,C)$ by a small prime number $p$
represented by $f$; the powers $f1 = f$, $f2$, $f3$ of $f$ then 
represent $p$, $p^2$, $p^3$ etc., and the exponent $m$ of the $m$-th 
power $fm$ is called the pointer (Zeiger\footnote{This word is apparently
borrowed from the book \cite{Ett} on combinatorial analysis by 
Andreas von Ettinghausen, professor of mathematics at the University
of Vienna. Ettinghausen used the word ``Zeiger'' (see \cite[p. 2]{Ett}) 
as the German translation of the Latin word ``index''. \Sim\ refers
to \cite{Ett} in \cite[p. 55]{Sim1}.}) of $f$. What we denote by 
$\cs(Q^m) \equiv a \bmod \cR$, \Sim\ wrote as $fm = a$. 

\Sim\ introduced this notation in \cite[Art. 10]{Sim1}; instead of
$\cs(Q) = 2$ for $Q = (2,0,c)$ he simply wrote $(2,0,d) = 2$.
He explained the general case as follows:
\begin{quote}
So ist z.B. $(180,-17,193) = \frac{3^2 \times 5}{2^2}$ weil 
$180 = 2^2 \times 3^2 \times 5$ und $-17 \equiv -1 \pmod 4$, 
$-17 \equiv 1 \pmod 6$, $-17 \equiv 3 \pmod {10}$.\footnote{Thus we 
have, for example, $(180,-17,193) = \frac{3^2 \times 5}{2^2}$ because
$180 = 2^2 \times 3^2 \times 5$ and $-17 \equiv -1 \pmod 4$, 
$-17 \equiv 1 \pmod 6$, $-17 \equiv 3 \pmod {10}$.}
\end{quote}

One of the tricks he used over and over again is the following: 
\begin{equation}\label{E11}
  (A,B,C) \sim (A, B \pm 2A, A \pm B + C) \sim (A \pm B + C, -B \mp 2A, A)
\end{equation}
shows that if $Q = (A,B,C)$ represents an integer 
$m = Q(1,-1) = A \pm B + C$, then $\cs(Q)$ can be computed from 
$Q \sim (m, \mp 2A - B, A)$. Similarly, we have
$$ (A,B,C) \sim (A \pm B + C, B \pm 2C, C). $$

\section{\Sim's Calculations}
In this section we will reconstruct a few of \Sim's calculations
of (factors of) class numbers and factorizations.

\subsection*{$\Delta = -10079$}
\Sim\ first considers a simple example (see \cite[p. 58]{Sim1}): 
he picks a discriminant $\Delta$ for which $\Delta + 1$ is divisibly 
by $2$, $3$, $5$ and $7$, namely $\Delta = -10079$. 
Consider the form $Q = (5,1,504)$ with discriminant $\Delta$. The small 
powers of $Q$ provide us with the following factorizations:
$$ \begin{array}{c|c|c}
   \rsp     n &   Q^n            &  \cs(Q^n) \\  \hline
   \rsp     1 & \sim(504,-1,5)   &  2^{-3} \cdot 3^{-2} \cdot 7^{-1} \\
   \rsp     3 & (36,17,72)       &  2^2 \cdot 3^{-2}   \\
              & \sim (72,-17,36) &  2^{-3} \cdot 3^2 
   \end{array} $$

This implies
\begin{align*}
  \cs(Q^6) & \equiv \cs(Q^3) \cs(Q^3) 
             \equiv 2^2 \cdot 3^{-2} \cdot 2^{-3} \cdot 3^2  \equiv 2^{-1}, \\
  \cs(Q^{15}) & \equiv \cs(Q^3)^3 \cs(Q^3)^2 
               \equiv 2^6 \cdot 3^{-6} \cdot 2^{-6} \cdot 3^4 \equiv 3^{-2}, \\
  \cs(Q^{32}) & \equiv \cs(Q^{-1})  \cs(Q^{-3}) \cs(Q^6)^6 \equiv 7. 
\end{align*}
Now $7 = \cs(R)$ for $R = (7,1,360)$: this is easily deduced from 
$\Delta \equiv 1 \equiv 1^2 \bmod 7$. From $R^2 \sim (49,-41,60)$
\Sim\ reads off $\cs(Q^{64}) \equiv 2^2 \cdot 3^{-1} \cdot 5$.
But then $\cs(Q^{63}) \equiv 2^2 \cdot 3^{-1}$ and therefore
$$ \cs(Q^{75}) \equiv \cs(Q^{63}) \cdot  \cs(Q^6)^2 
              \equiv 2^2 \cdot 3^{-1} \cdot 2^{-2} 
              \equiv 3 \bmod \cR. $$
This implies $\cs(Q^{150}) \equiv  \cs(Q^{15})$ and therefore
$\cs(Q^{135}) \equiv 1 \bmod \cR$. Since neither $Q^{45}$ nor $Q^{27}$
are principal, the class of $Q$ has order $135$.

For showing that $h(\Delta) = 135$, \Sim\ would have to determine
the pointers of all primes $p < \sqrt{-\Delta/3} \approx 100.3$. The 
fact that $h$ is odd would then also show that $\Delta$ is a prime 
number.

\subsection*{$\Delta = - 121271$}

For larger discriminants, \Sim\ suggests the following method:
\begin{quote}
Bei grossen Determinanten, oder wo die vorige Methode nicht zum
Ziele f\"uhrt, nimmt man die Zeiger einiger kleiner Primzahlen
als unbekannt an, scheidet dann jene Gr\"ossen  aus den Producten
der Bestimmungsgleichungen aus, und sucht die anderen Primzahlen
in Bestimmungsgleichungen durch jene unbekannten Zeiger 
dar\-zu\-stel\-len.\footnote{For large determinants, or in cases 
where the preceding method is not successful, we take the indices
of some small primes as unknowns, eliminates those numbers from 
the products of the determination equations, and seeks to represent
these unknown indices by the other primes in these determination
equations.}
\end{quote}

\Sim\ chooses the discriminant $\Delta = -121271$; in the course of
the calculation it becomes clear that $\Delta = 99^2 - 2^{17}$, and
quite likely the discriminant was constructed in this way. This
is supported by \Sim's remark on \cite[p. 64]{Sim1} that if $D = a^m - b^2$ is
a (positive) determinant and if $a$ is odd, then the exponent of
the form $(a,2b,a^{m-1})$ is divisible by $m$, as can be seen
from the ``period''
$$ (a,2b,a^{m-1}), (a^2,2b,a^{m-2}), \ldots, (a^m,2b,1). $$
Observe that this statement only holds under the additional assumption
that these forms be reduced, i.e., that $0 < 2b \le a$. Examples are
$D = 3^3 - 1 = 26$ and $h(-4 \cdot 26) = 6$, or
$D = 3^5 - 4 = 239$ and $h(-4 \cdot 239) = 15$.
A similar observation was made by Joubert \cite{Joub} just a few years
after \Sim. The connection between classes of order $n$ and solutions
of the diophantine equation $a^m - Dc^2 = b^2$ was investigated
recently in \cite{HL}.

\medskip

Let us write $Q_2 = (2, 1, 15159)$ and $Q_3 = (3,1,10106)$. Then
$Q_2^2 \sim (4,5,7581)$ and $\cs(Q_2^2) \equiv 3 \cdot 7^{-1} \cdot 19^{-2}$.
Since $\cs(Q_3) \equiv 3$, we find $\cs(Q_2^{-2} Q_3) \equiv 7 \cdot 19$.

$Q_2^3 \sim (8,13,3795)$ gives 
$\cs(Q_2^3) \equiv 3^{-1} \cdot 5 \cdot 11^{-1} \cdot 23$ 
and $\cs(Q_2^3 Q_3) \equiv 5 \cdot 11^{-1} \cdot 23$.

We can summarize \Sim's calculations as follows:

\begin{minipage}[t]{6.2cm}
$$ \begin{array}{r|c|c}
  \rsp   n  &   Q_2^n \sim & \cs(Q_2^n) \bmod \cR \\  \hline
  \rsp   2  & (4,5,7581)         &   \\
  \rsp      & (7581,-5,4 )  &  3 \cdot 7^{-1} \cdot 19^{-2} \\
  \rsp   3  & (8,13,3795)        &  \\
  \rsp      & (3795,-13,8)  & 3^{-1} \cdot 5 \cdot 11^{-1} \cdot 23 \\
  \rsp   4  & (16,29,1908)       & \\
  \rsp      & (1953,-61,16) & 3^{-2} \cdot 7^{-1} \cdot 31 \\     
  \rsp   5  & (32,29,954)        & \\
  \rsp      & (957,35,32)   &  3^{-1} \cdot 11 \cdot 29 \\
  \rsp      & (1015,-93,32) & 5^{-1} \cdot  7 \cdot 29 
\end{array} $$ \end{minipage}
\begin{minipage}[t]{6cm}
$$ \begin{array}{r|c|c}
  \rsp   n  &   Q_2^n \sim  & \cs(Q_2^n) \bmod \cR \\  \hline
  \rsp   6  & (64,29,477)        &  \\
  \rsp      & (477,-29,64) &  3^2 \cdot 53 \\
  \rsp      & (675,227,64)  & 3^{-3} \cdot 5^{-2} \\
  \rsp   7  & (128,157,285) &   \\
  \rsp      & (285,-157,128) & 3^{-1} \cdot 5 \cdot 19^{-1}  \\
  \rsp      & (483,355,128)  & 3^{-1} \cdot 7 \cdot 23^{-1}  
   \end{array} $$
\end{minipage}

Note that if $\cs(Q_2^n) \equiv 2^{-1}u$ for some odd number $u$, then 
$\cs(Q_2^{n+1}) \equiv u$. Thus $\cs(Q_2^4) \equiv 2^{-2} \cdot 3^2 \cdot 53$
implies $\cs(Q_2^6) \equiv 3^2 \cdot 53$, and in such cases we have listed
only the relation that does not involve a power of $2$.

The computation of $Q_2^7$ reveals $\Delta = 99^2 - 2^{17}$, and shows
that $\cs(Q_2^7) \equiv 2^{-8}$, which 
gives $\cs(Q_2^{15}) \equiv 1$.

Now \Sim\ continues as follows: the relations
$$ \cs(Q_2^2) \equiv  3 \cdot 7^{-1} \cdot 19^{-2} \quad \text{and} \quad
  \cs(Q_2^7) \equiv 3^{-1} \cdot 5 \cdot 19^{-1}$$ 
give 
$$ \cs(Q_2^{12}) \equiv \cs((Q_2^7)^2Q_2^{-2}) 
                \equiv 3^{-2} \cdot 5^2 \cdot 19^{-2} 
                 \cdot 3^{-1} \cdot 7 \cdot 19^{2}
               = 3^{-3} \cdot 5^2 \cdot 7. $$
Using the relations
$$ \cs(Q_2^{12} Q_3^3) \equiv 5^2 \cdot 7, \quad \text{and} \quad
   \cs(Q_2^6 Q_3^3) \equiv 5^{-2}, $$
\Sim\ deduces
\begin{equation}\label{ES7}
 \cs(Q_2^3 Q_3^6) \equiv \cs(Q_2^{18} Q_3^6) \equiv 7. 
\end{equation}

This allows him to eliminate the $7$s from his relations, which gives
\begin{align*}
 \cs(Q_2^{-4} Q_3^7) & \equiv \cs(Q_2^{-7}) \cs(Q_3) \cs(Q_2^3 Q_3^6) \equiv 23, \\
 \cs(Q_2^7 Q_3^8)   & \equiv \cs(Q_2^4) \cs(Q_3^2) \cs(Q_2^3 Q_3^6) \equiv 31.
\end{align*}
For the actual computation of the order of $Q_3$, only the relation
(\ref{ES7}) will be needed.

\Sim\ also investigates the powers of $Q_3$ and finds

\begin{minipage}[t]{6.2cm}
$$ \begin{array}{r|c|c}
  \rsp   n  &   Q_3^n \sim  & \cs(Q_3^n) \bmod \cR \\  \hline
  \rsp   1  & (3,1,10106)        &  \\
  \rsp      & (10108,2,3)   & 2^2 \cdot 7 \cdot 19^2 \\
  \rsp   3  & (27,43,1140)       &  \\ 
            & (1210,-97,27) & 2^{-1} \cdot 5 \cdot 11^{-2} \\
  \rsp      & (1162, 65,27) & 2 \cdot 7^{-1} \cdot 83 \\
  \rsp   4  & (81, 43, 380) &   \\
  \rsp      & (380,-43,81) &  2^2 \cdot 5^{-1} \cdot 19^{-1} \\
  \rsp      & (418,119,81) & 2^{-1} \cdot 11 \cdot 19 
  \end{array} $$ \end{minipage}
\begin{minipage}[t]{6cm}
$$ \begin{array}{r|c|c}
  \rsp   n  &   Q_3^n \sim  & \cs(Q_3^n) \bmod \cR \\  \hline
  \rsp   5  & (243,205,168) &  \\
  \rsp      & (616,541,168) & 2^3 \cdot 7^{-1} \cdot 11^{-1} \\
  \rsp   6  & (729,205, 56) & \\
  \rsp      & (56,-205,729) & 2^{-3} \cdot 7 
   \end{array} $$
\end{minipage}
\smallskip

\Sim\ observes
$$ \cs(Q_2^2 Q_3^9) \equiv \cs(Q_3^3) \cs(Q_2^{-1})  \cs(Q_2^3 Q_3^6) \equiv 83, $$
but does not use this relation in the sequel. He continues with
$$ \cs(Q_2 Q_3^4) \equiv 11 \cdot 19, \quad 
  \cs(Q_2^3 Q_3^{-5}) \equiv 7 \cdot 11, $$
from which he derives the following relations:
\begin{align*}
  \cs(Q_3^{-11})  & \equiv \cs(Q_2^3 Q_3^{-5})  \cs(Q_2^{-3} Q_3^{-6}) \equiv 11, &
  \cs(Q_2 Q_3^{15})   & \equiv \cs(Q_2 Q_3^4)\cs(Q_3^{11}) \equiv 19, \\
  \cs(Q_2^8 Q_3^{16}  & \equiv \cs(Q_2^7) \cs(Q_3)  \cs(Q_2 Q_3^{15}) \equiv 5, &
  \cs(Q_2^{22} Q_3^{35}) & \equiv \cs(Q_2^{16} Q_3^{32})  \cs(Q_2^6Q_3^3) \equiv 1.
\end{align*}
Raising the last relation to the $15$th power yields $\cs(Q_3^{525}) \equiv 1$. 
Checking that $Q_3^{75}$, $Q_3^{105}$ and $Q_3^{175}$ are not principal then 
shows that $Q_3$ has order $h = 525 = 3 \cdot 5^2 \cdot 7$. In fact, 
{\tt pari} tells us that this is the class number of $\Delta = -121271$.

\section{Class Number Calculations}
Let us remark first that \Sim\ does not compute class numbers
but rather the order of a given form in the class group. Note that 
this is sufficient for factoring the discriminant. \Sim\ is well
aware of the fact that his method only produces divisors of the 
class number: in \cite[art. 13]{Sim1}, he writes
\begin{quote}
Was die L\"ange $\theta$ anbelangt, sucht man $fm = 1$ zu erhalten, 
wo dann entweder $\theta = m$ oder ein Theiler von $m$ ist. Die
wichtigsten Glieder der Perioden sind die zu kleinen Primzahlen
geh\"origen Formen. Welches die gr\"osste Primzahl w\"are, deren 
Zeiger man kennen m\"usse, um vor Irrthum sicher zu sein, konnte
ich bis jetzt nicht ermitteln, jedenfalls ist sie kleiner als 
$\sqrt{D/3}$ bei den unpaaren, und als $2 \sqrt{D/3}$ bei den
paaren Formen, wahrscheinlich aber reichen dazu nur wenige 
Primzahlen hin.\footnote{As for the length $\theta$ of the period,
one tries to find $fm = 1$, and then either $\theta = m$, or $\theta$
is a divisor of $m$. The most important members of the period are 
those belonging to small prime numbers. I have not yet found what 
the smallest prime number is whose pointer must be known in order not
to commit an error; in any case it is smaller than $\sqrt{D/3}$ for
odd forms, and than  $2 \sqrt{D/3}$ for the even forms, but most
likely just a few prime numbers are sufficient.}
\end{quote}

In the example $\Delta = -121271$ above we have seen that the powers
of $Q_2$ only give a subgroup of order $15$ in the class group, whereas
the powers of $3$ include all forms representing the primes
$$ p = 2, 3, 5, 7, 11, 19, 23, 29, 31, 53, 83. $$
For verifying that $h(-121271) = 525$, one would have to find the 
pointers for the other primes $p$ with $(\Delta/p) = +1$ and 
$\Delta < 202$ as well, namely those of
$$ p = 47, 61, 73, 79, 89, \ldots, 197. $$
Since the pointers of all small primes are known, this is only a 
little additional work. The fact that the class number is odd then
implies that $-\Delta = 121271$ is a prime.

\subsection*{$\Delta = - 4 \cdot 265371653$}

Consider the forms 
$$ Q_3 = (3,2, 88457218), \quad Q_{11} = (11,10,24124698), \quad \text{and} \quad
   Q_{13} = (13,10, 20413206). $$
Using a computer it is easily checked that $Q_3 \sim Q_{11}^{5} Q_{13}^{-3}$,
but this relation was apparently not noticed by \Sim. It would follow
easily from 
\begin{align*}
 Q = Q_{11}^5 & = (6591, -6568, 41899),  & Q(0,1) & = 11 \cdot 13 \cdot 293, \\
 Q = Q_{13}^3 & = (2197, -2174, 121326), & 
     Q(1,-1) & = 3 \cdot 11 \cdot 13 \cdot 293, 
\end{align*}
but perhaps the prime $293$ was not an element of \Sim's factor base.

A computer also finds the following relations among the small powers
of these three forms:
\begin{align*}
  Q_{11}^{13} Q_{13}^{11}  & = (1058, 918, 251023);
                 & \cs( Q_{11}^{13} Q_{13}^{11} ) & \equiv 2 \cdot 23^{-2}, \\
 Q_3^{14} Q_{11}^{12} Q_{13} & = (529, -140, 501657);
                 & \cs( Q_3^{14} Q_{11}^{12} Q_{13} ) & \equiv 23^{-2}.  
\end{align*}
Composition shows that
\begin{align*}
 Q_3^{-14} Q_{11} Q_{13}^{10} & \equiv
   Q_{11}^{13} Q_{13}^{11} Q_3^{-14} Q_{11}^{-12} Q_{13}^{-1}  \\
  & = (1058, 918, 251023) (529, 140, 501657) = 
   (2, 918, 132791167), 
\end{align*}
and squaring yields   
 $$  Q_3^{-28} Q_{11}^2 Q_{13}^{20} \sim Q_0. $$
Similarly,
\begin{align*}
 Q_3^3 Q_{11}^{15} Q_{13}^{11} & = (16389, -16010, 20102), & 
   \cs(Q_3^3 Q_{11}^{15} Q_{13}^{11}) & \equiv 2 \cdot 19 \cdot 23^2, \\
 Q_3^{12} Q_{11}^{15} Q_{13}^8 & = (6859, 5028, 39611), & 
   \cs(Q_3^{12} Q_{11}^{15} Q_{13}^8) & \equiv 19^3, \\
\intertext{which implies}
 Q_3^{3} Q_{11}^{15} Q_{13}^{11} \cdot Q_{11}^{13} Q_{13}^{11}
  & \sim (19, 12, 13966931), & \cs(Q_3^{3} Q_{11}^{28} Q_{13}^{22}) & \equiv 19, 
\end{align*}
and so
$$ 1 \equiv \cs(Q_3^{3} Q_{11}^{28} Q_{13}^{22})^{3}/\cs(Q_3^{12} Q_{11}^{15} Q_{13}^8)
   \equiv \cs(Q_3^{-3} Q_{11}^{69} Q_{13}^{58}). $$

Eliminating $Q_3 \sim Q_{11}^{5} Q_{13}^{-3}$ from the relations
$$ Q_3^{-28} Q_{11}^2 Q_{13}^{20} \sim Q_3^{-3} Q_{11}^{69} Q_{13}^{58}  \sim Q_0 $$
then implies
$$ Q_{11}^{-138} Q_{13}^{104} \sim Q_0 \quad \text{and} \quad 
   Q_{11}^{54}  Q_{13}^{67} \sim Q_0,   $$ 
hence
$$ Q_{11}^{14862} \sim Q_0. $$
It is then easily checked that $Q_3$ and $Q_{11}$ have exponent $14862$ in
the class group, whereas $Q_{13}$ is a sixth power and has order
$2477$. A quick calculation with {\tt pari} reveals that $h(\Delta) = 14862$. 

\Sim\ must have proceeded differently, as he records the relations
$$ Q_3^{119} Q_{11}^{11} Q_{13}^8 \sim Q_0, \quad
   Q_3^{1276} Q_{11}^{94} Q_{13}^{26} \sim Q_0, \quad
   Q_3^{385} Q_{11}^{31} Q_{13}^4 \sim Q_0. $$
It is not impossible that by playing around with small powers of 
$Q_3$, $Q_{11}$ and $Q_{13}$, \Sim's calculations can be reconstructed.
It is more difficult to reconstruct \Sim's factorization of 
$N = \frac19(10^{17}-1)$, since he left no intermediate results at all
(apparently he was forced to shorten his manuscript drastically before
publication).

\Sim\ knew that it is often not necessary to determine the class number 
for factoring integers; in \cite[Art. 17]{Sim1} he observed:
\begin{quote}
Bei Zahlenzerlegungen nach dieser Methode findet man oft $f 2a = m^2$,
oder es l\"asst sich aus den Bestimmungsgleichungen eine solche Form
ableiten; dann hat man $\frac{f 2a}{m^2} = (\frac{fa}m)^2 = 1$, und
es kann $fa:m$ blos eine Schluss- oder Mittelform sein. Gew\"ohnlich
ist das letztere der Fall. \footnote{In factorizations with this method
one often finds $fa = m^2$, or such a form can be derived from certain
determination equations; then we have $\frac{f 2a}{m^2} = (\frac{fa}m)^2 = 1$,
and $fa:m$ can only be an end or a middle form. Most often, the latter 
possibility occurs.}
\end{quote}

To illustrate this idea we present an example that cannot be found in 
\Sim's article. Let $\Delta =  -32137459$ and consider the form
$Q = (5,1, 1606873)$ with discriminant $\Delta$. It is quickly seen
that $Q^{26}(1,0) = 11^2$. This observation immediately leads to a 
factorization of $\Delta$: the form $Q^{26}$ represents $11^2$, 
hence $Q^{13}$  represents $11$, as does $Q_{11} = (11, 3, 730397)$. 
Thus $(Q^{13}R^{-1})^2$ represents $1$, which implies that $Q^{13}R^{-1}$ 
is ambiguous (see \cite[S. 36]{Sim1}). In fact, 
$Q^{13}R^{-1} = (1511, 1511, 5695)$, which gives the factorization
$\Delta = - 1511 \cdot 21269$.

\section{Shanks}

The factorization method based on the class group of binary quadratic 
forms was rediscovered by Shanks \cite{Sha}, who, however, used a 
completely different method for computing the class group: he estimated 
the class number $h$ using truncated Dirichlet L-series and the found 
the correct value of $h$ with his baby step -- giant step method. 
Attempts of speeding up the algorithm led, within just a few years, 
to Shanks's discovery of the infrastructure and his square form 
factorization method SQUFOF.

The factorization method described by \Sim\ was rediscovered by
Schnorr \cite{Schn};  the \Sim\ map is defined in \cite[Lemma 4]{Schn} 
(see also \cite[Thm. 3.1]{Sey}), although in a slightly different guise: 
a quadratic form $Q = (a,b,c)$ is factored into ``prime forms''
$I_p = (p,b_p,C)$, where $B = b_p$ is the smallest positive solution of
$B^2 \equiv \Delta \bmod 4p$  for $\Delta = -N \equiv 1 \bmod 4$.
Thus the equation corresponding to our
$$ \cs(Q) = \prod_{i=1}^n p_i^{\pm e_i} \quad \text{looks like} \quad 
   Q = \prod_{i=1}^n (I_p)^{\pm e_i} $$
in \cite{Sey}, ``where the plus sign in the exponent $e_i$ holds if 
and only if $b \equiv b_{p_i} \bmod 2p_i$. Variations of this method
were later introduced by Mc Curley and Atkin.

\Sim's method is superior to Schnorr's for calculations by hand since 
it allows him to use the factorizations of $Q(0,1)$ and $Q(1,\pm 1)$.
The main difference between the two methods is that \Sim\ factors the
forms $Q_p^n$ for small prime numbers $p$ and small exponents $n$,
 whereas Schnorr factors products $Q_1^{n_1} \cdots Q_r^{n_r}$ of forms
$Q_j = (p_j,*,*)$ for primes in his factor based and exponent vectors
$(n_1,\ldots, n_r)$ chosen at random.

\Sim's question in Section 4 concerning the number of primes $p$ 
such that the forms $(p,B,C)$ generate the class group was answered
under the assumption of the Extended Riemann Hypothesis by Schoof
\cite[Cor. 6.2]{Schoof}, who showed that the first $c \log^2|\Delta|$ 
prime numbers suffice; Bach \cite{Bach} showed that, for fundamental
discriminants $\Delta$, we can take $c = 6$.

The basic idea of combining relations, which is also used in factorization
methods based on continued fractions, quadratic sieves or the number
field sieve, is not due to \Sim\ but rather occurs already in the work
of Fermat and played a role in his challenge to the English mathematicians,
notably Wallis and Brouncker. In this challenge, Fermat explained that
if one adds to the cube $343 = 7^3$ all its proper divisors, then the
sum $1 + 7 + 7^2 + 7^3 = 400 = 20^2$ is a square, and asked for another
cube with this property. 

Fermat's solution is best explained by studying a simpler problem
first, namely that of finding a number $n$ with $\sigma(n^2) = m^2$, 
where $\sigma(n) = \sum_{d \mid n} 1$ is the sum of all divisors of a 
number. Making a table of $\sigma(p)$ for small prime powers $p$ one 
observes that $\sigma(2^4) = \sigma(5^2) = 31$, hence $\sigma(20^2) = 31^2$.

The solution\footnote{Sufficiently many hints can be found in Frenicle's 
letter in \cite[XXXI]{Wallis}, and in subsequent letters by Wallis and 
Schooten. See also the detailed exposition given by Hofmann \cite{Hof}.} 
of Fermat's challenge also exploits the multiplicativity of $\sigma(n)$: 
with little effort one prepares a table for the 
values of $\sigma(p)$ for  small primes $p$ such as the following:

$$ \begin{array}{r|c} 
    p  & \sigma(p^3) \\ \hline 
    2  & 3 \cdot 5 \\
    3  & 2^3 \cdot 5 \\
    5  & 2^2 \cdot 3 \cdot 13 \\
    7  & 2^4 \cdot 5^2 \\
   11  & 2^3 \cdot 3 \cdot 61 
  \end{array} \qquad \qquad
  \begin{array}{r|c} 
    p  & \sigma(p^3) \\ \hline 
   13  & 2^2 \cdot 5 \cdot 7 \cdot 17 \\
   17  & 2^2 \cdot 3^2 \cdot 5 \cdot 29 \\
   19  & 2^3 \cdot 5 \cdot 181 \\
   23  & 2^4 \cdot 3 \cdot 5 \cdot 53 \\
   29  & 2^2 \cdot 3 \cdot 5 \cdot 421 
  \end{array} \qquad \qquad
  \begin{array}{r|c} 
    p  & \sigma(p^3) \\ \hline 
   31  & 2^6 \cdot 13 \cdot 37 \\
   37  & 2^2 \cdot 5 \cdot 19 \cdot 137 \\
   41  & 2^2 \cdot 3 \cdot 7 \cdot 29^2 \\
   43  & 2^3 \cdot 5^2 \cdot 11 \cdot 37 \\
   47  & 2^5 \cdot 3 \cdot 5 \cdot 13 \cdot 17 
   \end{array} $$

\smallskip \noindent
Then it is readily seen that 
$n = 751530 = 2 \cdot 3 \cdot 5 \cdot 13 \cdot 41 \cdot 47$.

\section*{Concluding Remarks}

\Sim's contributions to the theory of quadratic forms and the
factorization of numbers would have remained unknown if his articles
could not be found online. In particular, his memoirs \cite{Sim1,Sim2,Sim59}
can be accessed via google books\footnote{See {\tt http://books.google.com}}, 
and the articles that appeared in the journal \v{C}asopis are available
on the website of the GDZ\footnote{see 
{\tt http://gdz.sub.uni-goettingen.de/dms/load/toc/?PPN=PPN31311028X}}
in G\"ottingen. I would also like to remark that a prerequisite for 
understanding the importance of \cite{Sim1} is a basic familiarity 
with composition of binary quadratic forms.

I do not know where \Sim\ acquired his knowledge of number theory. 
\Sim\ was familiar with Legendre's ``Essais de Th\'eorie des Nombres''
and Gauss's ``Disquisitiones Arithmeticae'', as well as with publications
by Scheffler \cite{Sch} on diophantine analysis\footnote{This is an
interesting book, which contains not only the basic arithmetic of
the integers up to quadratic reciprocity, but also topics such as
continued fractions in Gaussian integers, which are discussed using
geometric diagrams, and the quadratic reciprocity law in $\Z[i]$.}, 
and by Dirichlet \cite{Dir} and Lipschitz \cite{Lip} 
on the class number of forms with  nonsquare discriminants. Since 
Lipschitz's article appeared in 1857, \Sim\ must have had access to 
Crelle's Journal while he was teaching in Budweis.

\Sim's article \cite{Sim1} contains other ideas that we have not discussed.
In particular, in \cite[Art. 12]{Sim1} he tries to get to grips with 
decompositions of noncyclic class groups into ``periods'' (cyclic
subgroups); in this connection he gives the example $\Delta = -2184499$
with class group of type\footnote{\Sim\ remarks that this is a ``remarkably
rare case''. In fact, the smallest discriminant with a noncyclic 
$5$-class group is $\Delta = -11199$, and the minimal $m$ with 
$\Delta = -4m$ and noncyclic $5$-class group is $m = 4486$.}  $(5,5,11)$. 
In \cite[Art. 18]{Sim1}, \Sim\ solves diophantine equations of the form 
$pz^m = ax^2 + bxy +cy^2$.

\bigskip

\end{document}